\documentclass[12pt]{article}
\pdfoutput=1

\usepackage{amsmath}
\usepackage[utf8]{inputenc}
\usepackage[UKenglish]{babel}
\usepackage [autostyle]{csquotes}
\MakeOuterQuote{"}
\usepackage[normalem]{ulem} 
\usepackage{aliascnt}
\usepackage{amssymb}
\usepackage{amsthm}
\usepackage{faktor}
\usepackage{mathtools}
\usepackage{amsfonts}
\usepackage{enumerate}
\usepackage{enumitem}
\usepackage{multicol}
\usepackage{float}
\usepackage{tikz}
\usepackage{tikz-cd} 
\usepackage{comment}
\usepackage[round,comma,sort
]{natbib}
\usepackage{hyperref}
\hypersetup{
    colorlinks=true,       
    linkcolor=blue,          
    citecolor=olive,       
    filecolor=magenta,      
    urlcolor=blue          
}
\usepackage[capitalise, nameinlink]{cleveref}
\crefname{subsection}{subsection}{subsections}
\newcommand{\Z}{\mathbb{Z}}
\newcommand{\ga}{\Gamma}

\newcommand{\gp}[2]{\langle #1 \, | \, #2 \rangle}
\newcommand{\sgp}[1]{\langle #1 \rangle}

\DeclareMathOperator{\lk}{Link}

\def\coloneqq{\mathrel{\mathop\mathchar"303A}\mkern-1.2mu=}

\begin{document}

\title{On Subgroup Separability and Membership Problems\\ in Twisted Right-Angled Artin Groups}
\author{Islam Foniqi}

%\date{\today}
\date{}
\maketitle

% ------- Theorem styles -------
\theoremstyle{plain}
\newtheorem{theorem}{Theorem}[section]
\newtheorem*{theorem*}{Theorem}
\newtheorem*{prop*}{Proposition}

\newtheorem{conj}[theorem]{Conjecture}
\newtheorem{notation}[theorem]{Notation}
\newtheorem*{note}{Note}
% definition, example
\theoremstyle{definition}

\setlength{\parindent}{0em}
\setlength{\parskip}{0.5em} 
\author{}

\newaliascnt{conjecture}{theorem}
\newtheorem{conjecture}[conjecture]{Conjecture}
\aliascntresetthe{conjecture}
\providecommand*{\conjectureautorefname}{Conjecture}

\newaliascnt{lemma}{theorem}
\newtheorem{lemma}[lemma]{Lemma}
\aliascntresetthe{lemma}
\providecommand*{\lemmaautorefname}{Lemma}

\newaliascnt{cor}{theorem}
\newtheorem{cor}[cor]{Corollary}
\aliascntresetthe{cor}
\providecommand*{\corautorefname}{Corollary}

\newaliascnt{claim}{theorem}
\newtheorem{claim}[claim]{Claim}

\newaliascnt{notation}{theorem}
\aliascntresetthe{notation}
\providecommand*{\notationautorefname}{Notation}

\aliascntresetthe{claim}
\providecommand*{\claimautorefname}{Claim}

\newaliascnt{remark}{theorem}
\newtheorem{remark}[remark]{Remark}
\aliascntresetthe{remark}
\providecommand*{\remarkautorefname}{Remark}

\newtheorem*{claim*}{Claim}
\theoremstyle{definition}

\newaliascnt{definition}{theorem}
\newtheorem{definition}[definition]{Definition}
\aliascntresetthe{definition}
\providecommand*{\definitionautorefname}{Definition}

\newaliascnt{example}{theorem}
\newtheorem{example}[example]{Example}
\aliascntresetthe{example}
\providecommand*{\exampleautorefname}{Example}

\newaliascnt{prop}{theorem}
\newtheorem{prop}[prop]{Proposition}
\aliascntresetthe{prop}
\providecommand*{\propautorefname}{Proposition}

\newaliascnt{question}{theorem}
\newtheorem{question}[question]{Question}
\aliascntresetthe{question}
\providecommand*{\questionautorefname}{Question}

\begin{abstract}
We characterize twisted right-angled Artin groups (T-RAAGs) that are subgroup separable using only their defining mixed graphs: such a group is subgroup separable if and only if the underlying simplicial graph contains neither induced paths nor squares on four vertices. This generalizes the results of Metaftsis–Raptis on classical right-angled Artin groups. Additionally, we show that the subgroup membership problem is decidable when the group is coherent, which occurs precisely when the defining mixed graph is chordal. We also address the rational and submonoid membership problems by exhibiting a cone-family of graphs for which the corresponding T-RAAGs have decidable rational and submonoid membership problems.
\end{abstract}

\renewcommand{\thefootnote}{\fnsymbol{footnote}} 

\thefootnote{
}

\footnotetext{The author acknowledges support from the EPSRC Fellowship grant EP/V032003/1 ‘Algorithmic, topological and geometric aspects of infinite groups, monoids and inverse semigroups’.
}

%%%%%%%%%%%%%%%%%%%%%%%%%%%%%%%%%%%%%%%%%%%%%%%%%%%%%%%%%%%%%%%%%%%%%%%%%%%%%%%%%%%%%%%%%%%%%%%%%%%%%%%%

\section{Introduction}\label{Introduction}
The purpose of this paper is to describe twisted right-angled Artin groups that are subgroup separable, just by examining their defining graphs.

Subgroup separability is a compelling property of groups with significant connections with geometric topology, see \citep{scott1978subgroups} for an overview. For historical reasons, subgroup separability is also called LERF (for locally extended residually ﬁnite).
Being subgroup separable is inherited by subgroups and finite extensions, and is preserved by taking free products; however, it is not preserved by taking direct products. Nevertheless, a direct product (and a semi-direct product) of a subgroup separable group with the infinite cyclic group is again subgroup separable. This will be a very useful tool for our work in this article. 

In geometric group theory, the class of
%, the class of \emph{Artin groups}, and in particular, the subclass of 
\emph{right-angled Artin groups} (\emph{RAAGs}) plays a major role, see \citep{charney2007introduction} for a survey. RAAGs are also known as \emph{graph groups}. \\
\cite{metaftsis2008profinite} show that one can decide if a RAAG is subgroup separable or not by just examining its defining graph; the only obstructions for a RAAG to be subgroup separable are the two well known examples of non-subgroup separable groups, which are the RAAGs on~$4$ generators based on a square, and a path; the first one is~$F_2 \times F_2$, and the second one is denoted by~$L$ in their work.

This article deals with the same problem, with focus in another class of groups that generalises the class of RAAGs. 
This is the class of \emph{twisted right-angled Artin groups} (\emph{T-RAAGs} shortly). T-RAAGs appear in \citep{pride1986tits}, and also in  \citep{clancy2010homology} -- where the notion of a \emph{twisted Artin group} was introduced. We refer to T-RAAGs also with the name \emph{twisted graph groups}. More details regarding their normal forms and other properties can be found in the PhD thesis \citep[Chapter~3]{foniqi2022}, and also in \citep{foniqi2024twisted}. 

The presentation of T-RAAGs is given by generators and relations: there are finitely many generators and there is at most one relation between any two distinct generators~$a,b$; this relation can be either a commutation~$ab = ba$, or a so-called {\it Klein relation}~$aba = b$. This presentation can be encoded by the use of mixed graphs, where vertices represent the generators of the group, while edges give rise to the relations. An undirected edge, connecting generators~$a$ and~$b$, gives rise to the commutation~$ab = ba$, while a directed edge with origin at~$a$ and terminus at~$b$ defines the Klein relation~$aba = b$. If generators~$a,b$ are not connected by an edge, they are free of relations. \\
As basic and important examples of T-RAAGs, we have both the fundamental groups of a torus:~$\Z^2 = \gp{a, b}{ab = ba}$, and of the Klein bottle:~$K = \gp{a, b}{aba = b}$. Note that~$K$ is not a RAAG (because RAAGs are bi-orderable but~$K$ is not).

\Cref{Definitions_preliminaries_and_notation} serves to set the notation and provide the main definitions for the whole article. \\
Next, in \Cref{Subgroup separability for RAAGs}, we provide a summary of results to characterise RAAGs that are LERF, given by \autoref{thm: main theorem for RAAGs}, and proved by \cite{metaftsis2008profinite}:
%there is a criterion for the subgroup separability for RAAGs, where it is shown that 
a RAAG~$A(\Gamma)$ is LERF if and only if~$\Gamma$ does not contain both the path~$P_4$ and the square~$C_4$, as induced subgraphs.

In \Cref{Going from RAAGs to T-RAAGs} we provide our main contribution, which is the following result:

\begin{theorem*}[\autoref{thm: main theorem}]
Let~$\Gamma$ be a mixed graph and~$T(\Gamma)$ the T-RAAG based on~$\Gamma$. The group~$T(\Gamma)$ is subgroup separable if and only if~$\Gamma$ does not contain, as induced subgraphs, mixed paths and mixed squares in~$4$ vertices.
\end{theorem*}

%Our main work, outlined in \Cref{Going from RAAGs to T-RAAGs}, characterises T-RAAGs that are subgroup separable. This characterisation shows that a T-RAAG is subgroup separable, if and only if its defining mixed graph does not contain paths or squares in~$4$ vertices. 
Here, in contrast with RAAGs, some of the edges in our graphs can be directed. In the case when the graph has no directed edges, we recover the result of \cite{metaftsis2008profinite} for RAAGs.

In \Cref{Membership Problem in T-RAAGs} we discuss membership problems in T-RAAGs. Regarding membership in finitely generated subgroups we have the following:

\begin{prop*}[\autoref{prop: subgroup membership}] Let $\Gamma$ be a mixed chordal graph. Then,~$T(\Gamma)$ has decidable subgroup membership problem.
\end{prop*}

Non-separable T-RAAGs have undecidable submonoid and rational subset membership problems. On the other hand, we construct a cone-family $\mathcal{R}$ of mixed graphs (see \autoref{def: sink-cone} and \autoref{def: class R}) such that the corresponding T-RAAGs are subgroup separable, satisfying the following:

\begin{theorem*}[\autoref{thm: class R}]
Let $\Gamma$ be a mixed graph belonging to $\mathcal{R}$. Then $T(\Gamma)$ has decidable rational subset membership problem. Consequently, $T(\Gamma)$ has decidable submonoid membership problem as well.
\end{theorem*}

\section{Definitions, preliminaries, and notation}\label{Definitions_preliminaries_and_notation}

This section serves to set the notation for the article. Our groups will be given by finite presentations~$G = \gp{A}{R}$; an element~$g \in G$ can be written as a word in~$(A \cup A^{-1})^{\ast}$. If~$X$ is a subset of~$G$ then~$\sgp{X}$ denotes the subgroup of~$G$ generated by~$X$. 

\subsection{Right-angled Artin groups}\label{Right-angled Artin groups}
The graphs we use to define graph groups are called \emph{simplicial}, they come with no loops, and no multiple edges.
\begin{definition}\label{def_simplicial_graph}
A {\it simplicial graph~$\Gamma$} is a pair~$\Gamma = (V, E)$, where~$V = V\ga$ is a non-empty set whose elements are called \emph{vertices}, and~$E = E\ga \subseteq \{\{x,y\} \mid x,y \in V, x\neq y\}$ is a set of paired distinct vertices, whose elements are called \emph{edges}.
\end{definition}
For a finite simplicial graph~$\Gamma = (V, E)$ we use~$A(\Gamma)$ to denote the right-angled Artin group based on~$\Gamma$, defined by the presentation 
$$A(\Gamma) \coloneqq \langle\, V \mid uv = vu \text{ whenever } \{u,v\}\in E\,\rangle.$$
Two important examples of RAAGs,
% for our work
are those defined by paths and cycles on four vertices, provided graphically as:
\begin{figure}[H]
\centering
\begin{tikzpicture}[>={Straight Barb[length=7pt,width=6pt]},thick, scale =0.75]
\draw[] (-6.5, 0) node[left] {$P_4 =~$};
%% vertices & nodes
\draw[fill=black] (-6,0) circle (1pt) node[above] {$a$};
\draw[fill=black] (-5,0) circle (1pt) node[above] {$b$};
\draw[fill=black] (-4,0) circle (1pt) node[above] {$c$};
\draw[fill=black] (-3,0) circle (1pt) node[above] {$d$};
\draw[] (-3, 0) node[right] {$\,\,\, ,$};

%%% edges
\draw[thick] (-6,0) -- (-5,0);
\draw[thick] (-5,0) -- (-4,0);
\draw[thick] (-4,0) -- (-3,0);

\draw[] (1, 0) node[left] {$C_4 =~$};
%% vertices & nodes
\draw[fill=black] (2,1) circle (1pt) node[left] {$a$};
\draw[fill=black] (4,1) circle (1pt) node[right] {$b$};
\draw[fill=black] (4,-1) circle (1pt) node[right] {$c$};
\draw[fill=black] (2,-1) circle (1pt) node[left] {$d$};

%%% edges
\draw[thick] (2, 1) -- (4, 1);
\draw[thick] (4,1) -- (4,-1);
\draw[thick] (4,-1) -- (2,-1);
\draw[thick] (2,-1) -- (2,1);
\end{tikzpicture}
\end{figure} 
Some authors refer to~$P_4$ and~$C_4$ with the word \emph{poisonous}; similarly the corresponding groups~$L = A(P_4)$ and~$A(C_4) \simeq F_2 \times F_2$ get called \emph{poisonous groups}, as they tend to be obstructions for several properties.

\subsection{Twisted right-angled Artin groups}\label{Twisted right-angled Artin groups}

To define T-RAAGs we use mixed graphs, which are similar to simplicial graphs but directed edges are allowed.
%; a directed edge with origin at~$a$ and terminus at~$b$ defines the Klein relation~$aba = b$.
We define mixed graphs using simplicial graphs as an underlying structure.

%We encode the definition of T-RAAGs on a type of graphs called {\it mixed graphs}, which are similar to simplicial graphs but directed edges are allowed. 
%To define mixed graphs, we use simplicial graphs as an underlying structure.
%The {\it right-angled Artin group} based on a simplicial graph~$\Gamma$ is defined by the following presentation:
%$$G{\Gamma} = \langle\, V \mid ab = ba \text{ whenever } \{a,b\} \in E \, \rangle.$$
%We refer to~$\Gamma$ as the {\it defining graph} of~$G{\Gamma}$.

%\section{Twisted right-angled Artin groups}\label{section_traags}

\begin{definition}\label{def_mixed_graph}
A {\it mixed graph}~$\Gamma$ consists of an {\it underlying simplicial graph}~$(V, E)$, a set of {\it directed edges}~$D \subseteq E$, and two maps~$o, t \colon D \longrightarrow V$. We denote it as~$\Gamma = (V, E, D, o, t)$. \\
For an edge~$e=\{x,y\} \in D$, the maps~$o, t$ satisfy~$o(e), t(e) \in \{x,y\}$, and~$o(e) \neq t(e)$. Refer to~$o(e)$ and~$t(e)$ as the {\it origin} and the {\it terminus} of edge~$e$ respectively.
\end{definition}

\begin{notation} 
Let~$\Gamma = (V, E, D, o, t)$ be a mixed graph.
\begin{itemize}[itemsep=4pt,parsep=0pt,topsep=0pt, partopsep=0pt]
    \item[(i)] If~$e = \{a,b\} \in E \setminus D$, we write~$e = [a, b]$, and note that also~$e = [b, a]$. Instead,
    \item[(ii)] if~$e = \{a,b\} \in D$, we write~$e = [o(e), t(e)\rangle$. In this case, either~$e = [a, b\rangle$, or~$e = [b, a\rangle$.
\end{itemize}
Graphically, we present the respective edges~$[a,b]$, and~$[a,b\rangle$ as:
\begin{figure}[H]
\centering
\begin{tikzpicture}[>={Straight Barb[length=7pt,width=6pt]},thick]
%% vertices & nodes

\draw[fill=black] (0,0) circle (1.5pt) node[left] {$a$};
\draw[fill=black] (2,0) circle (1.5pt) node[right] {$b$};
\draw[fill=black] (5,0) circle (1.5pt) node[left] {$a$};
\draw[fill=black] (7,0) circle (1.5pt) node[right] {$b$};
%%% edges
\draw[thick] (0,0) -- (2,0);
\draw[thick, ->] (5,0) --  (7,0);
%\draw[thick] (0,0) -- node[below] {$(i)$} (2,0);
%\draw[thick, ->] (4,0) -- node[below] {$(ii)$} (6,0);
\end{tikzpicture}
%\caption{Types of edges in~$E\Gamma$}\label{types_of_edges}
\end{figure}
Since every edge in~$E = E\Gamma$ has exactly one of these two types, one can write: 
\[
E\Gamma = \overline{E\Gamma} \sqcup \overrightarrow{E\Gamma}, \text{ where } \overline{E\Gamma} = E \setminus D, \text{ and } \overrightarrow{E\Gamma} = D.
\]
\end{notation}
%\begin{mydef}
%Let~$\Gamma = (V, E, D, o, t)$ be a mixed graph. The simplicial underlying graph~$\overline{\Gamma} = (V, E)$ is called the {\it underlying graph} of~$\Gamma$.
%\end{mydef}
When dealing with group presentations encoded from graphs, we want edge~$[a,b]$ to represent the commutation of~$a,b$; and edge~$[a,b\rangle$ to represent the corresponding Klein relation; hence by a slight abuse of notation, we also set:~$[a,b] = aba^{-1}b^{-1}$, and~$[a,b\rangle = abab^{-1}$.

\begin{definition}\label{def_of_traags}
Let~$\Gamma = (V,E)$ be a mixed graph with~$E = \overline{E\Gamma} \sqcup \overrightarrow{E\Gamma}$. Define a group:
$$T(\Gamma) = \langle  V \mid ab = ba \text{ if~$[a,b]\in \overline{E\Gamma}$ }, aba = b \text{ if~$[a,b\rangle \in \overrightarrow{E\Gamma}$ }\rangle.$$
Call~$T(\Gamma)$ the {\it twisted right-angled Artin group} based on~$\Gamma$, and~$\Gamma$ {\it the defining graph} of~$T(\Gamma)$. 
\end{definition}
If~$\overrightarrow{E\Gamma} = \emptyset$ (or equivalently if~$E\Gamma = \overline{E\Gamma}$), then~$T(\Gamma)$ is the {\it right-angled Artin group}~$A(\Gamma)$.
\begin{remark}\label{rem: underlying RAAG of a T-RAAG}
For a given mixed graph~$\Gamma = (V, E, D, o, t)$, denote by~$\overline{\Gamma}$ the underlying simplicial graph. Call the group~$A(\overline{\Gamma})$, the \emph{underlying RAAG} of~$T(\Gamma)$. 
As a general theme, one tries to find connections between~$T(\Gamma)$ and~$A(\overline{\Gamma})$.
\end{remark}

\subsection{Subgroup separability}
\begin{definition}
Let~$G$ be a group and~$H \leqslant G$ a subgroup. We say that~$H$ is \emph{separable} in~$G$, if for any~$g \in G \setminus H$, there is a finite index subgroup~$K_g \leqslant G$ such that~$H \leqslant K_g$ but~$g \notin K_g$.
\end{definition}

\begin{definition}
A group~$G$ is called \emph{subgroup separable} (or LERF)
%(or LERF - \emph{locally extended residually ﬁnite}) 
if every ﬁnitely generated subgroup of~$G$ is separable in~$G$; equivalently, a ﬁnitely generated subgroup of~$G$ is equal to an intersection of subgroups of ﬁnite index in~$G$. 
\end{definition}

Being LERF is inherited by subgroups and finite extensions. Indeed, one has:
\begin{theorem}[{\citealp[Lemma 1.1]{scott1978subgroups}}]\label{thm: SgSep with subgroups and finite extensions}
Let~$G$ be a group and~$H \leqslant G$. Then:
\begin{itemize}[topsep=0pt,itemsep = 0pt]
\item If~$H$ is not LERF then~$G$ is not LERF.
\item If~$H$ is LERF and~$H \leqslant G$ has finite index, then~$G$ is LERF.
\end{itemize}
\end{theorem}
Being LERF is also preserved by taking free products.
\begin{theorem}[{\citealp[Corollary 1.2]{burns1971finitely}}]\label{thm: SgSep with free products}
If both~$G_1$ and~$G_2$ are LERF groups, then their free product~$G_1 \ast G_2$ is LERF as well.
\end{theorem}

In general, subgroup separability is not preserved by semi-direct products; indeed, the group~$F_2 \times F_2$ is not LERF, see \autoref{cor: A(C_4) is not LERF}. 
However, if the normal factor is ERF, then it becomes true. The notion of ERF is stronger than LERF. While LERF is defined with respect to finitely generated subgroups, ERF shares the same definition but with respect to all subgroups. Hence, if all subgroups of a group~$G$ are finitely generated, then~$G$ is LERF if and only if~$G$ is ERF. This is true for the group~$\Z$, which yields several useful results.
\begin{theorem}[{\citealp[Theorem 4]{allenby2006locally}}]\label{thm: SgSep with semi-direct products}
Let~$G = N\cdot H$ with~$N \cap H = \{1\}$ be a splitting extension of the normal finitely generated subgroup~$N$ by the group~$H$. If~$N$ is ERF and~$H$ is LERF, then~$G$ is LERF.
\end{theorem}
The infinite cyclic group~$\Z$ is ERF, as it is LERF and its subgroups are finitely generated.
\begin{cor}\label{cor: SgSep with semi-direct products}
If~$H$ is LERF, then any semi-direct product~$G = \Z \rtimes H$ is LERF as well.
\end{cor}
Since direct products are in particular semi-direct products, we have:
\begin{cor}\label{cor: SgSep with direct products}
If~$H$ is LERF, then the direct product~$G = \Z \times H$ is LERF as well.
\end{cor}

\subsection{Reidemeister-Schreier procedure}\label{R-S}

There is an algorithm, the {\it Reidemeister-Schreier procedure}, that given a presentation of a group~$G$ and sufficient information about a subgroup~$H$ of~$G$, provides a presentation for the subgroup~$H$, see \citep[Section 4 of Chapter II]{lyndon1977combinatorial}. 
%In our case we need to find a presentation for a normal subgroup~$H$ of a finitely presented group~$G$, with cyclic quotient~$G / H$.
\begin{definition}
Let~$F = F(X)$ be a free group with basis~$X$ and~$K$ a subgroup of~$F$. A {\it Schreier transversal} for~$K$ in~$F$, is a subset~$T$ of~$F$ which is a {\it right transversal}, i.e.
\begin{itemize}[itemsep=4pt,parsep=0pt,topsep=0pt, partopsep=0pt]
	\item the union~$\underset{t \in T}{\bigcup}Kt$ is equal to~$F$, and
	\item $Kt_1 \neq Kt_2$ for~$t_1 \neq t_2$ in~$T$, 
\end{itemize}
and furthermore, any initial segment (or prefix) of an element of~$T$ belongs to~$T$ as well. 
\end{definition}

In \cite{lyndon1977combinatorial}, the procedure is given like this:
\begin{itemize}[itemsep=4pt,parsep=0pt,topsep=0pt, partopsep=0pt]
	\item the group~$G$ is given by a presentation~$\gp{X}{R}$.
	\item $\pi: F(X) \to G$ is the canonical epimorphism from the free group~$F = F(X)$ onto~$G$. 
	\item $T$ is a Schreier transversal for~$K = \pi^{-1}(H)$ in~$F(X)$.
	\item The map~$\; \bar{\;}\colon F \to T$ maps~$w \in F$ to its coset representative~$\overline{w} \in T$.
\end{itemize}

Then,~$Y = \{tx(\overline{tx})^{-1}\mid t\in T,\; x\in X,\; tx \not\in T \}$ gives a set of generators for~$H$.

Furthermore, there is a map~$\tau \colon F(X) \to F(Y)$ that rewrites each word~$w\in \pi^{-1}(H)$ in terms of generators of~$Y$. Define
$$S = \{\tau(trt^{-1})\mid t\in T, r\in R\}.$$
Then~$\gp{Y}{S}$ is a presentation for~$H$. We refer to~$T$ as Schreier transversal for~$H$ in~$G$.

\section{Subgroup separability for RAAGs}\label{Subgroup separability for RAAGs}
In this section we review some properties of RAAGs, and in the next two subsections we provide a summary of the proof of the following result.
\begin{theorem}[{\citealp[Theorem 2]{metaftsis2008profinite}}]\label{thm: main theorem for RAAGs}
A RAAG~$A(\Gamma)$ is subgroup separable if and only if~$\Gamma$ does not contain~$P_4$ or~$C_4$ as induced subgraphs; equivalently, if and only if~$A(\Gamma)$ does not contain~$A(P_4)$ or~$A(C_4)$ as subgroups.
\end{theorem}
\begin{comment}
With our previous notation, this means
\begin{equation}\label{eq: RAAGs - main_equation}
\mathcal{A}(SgSep) = \mathcal{A} \setminus \mathcal{A}(P_4 \lor C_4) = \mathcal{A}(\mathcal{E}),
\end{equation}
which we show in the next two subsections.
\end{comment}
\subsubsection*{RAAG subgroups via graphs}
In this part, we only deal with the second equivalence in \autoref{thm: main theorem for RAAGs}, namely, showing that: \\
{\bf Goal.} A RAAG~$A(\Gamma)$ contains~$A(P_4)$ or~$A(C_4)$ as subgroups, if and only if~$\Gamma$ contains~$P_4$ or~$C_4$ respectively, as induced subgraphs.

One direction of our goal is clear, because if~$\Delta$ is an induced subgraph of the simplicial graph~$\Gamma$, then~$A(\Delta)$ is a subgroup of~$A(\Gamma)$. \\
As a converse to the statement above, one has the following two results: 
\begin{itemize}[topsep=0pt,itemsep = 0pt]
\item[$\bullet$] \cite[Corollary 3.8]{kambites2009commuting} shows that if~$A(\Gamma)$ contains a subgroup isomorphic to~$A(C_4)$, then~$\Gamma$ contains an induced square~$C_4$.
\item[$\bullet$] \cite[Theorem 1.7]{kim2013embedability} shows that if~$A(P_4) \xhookrightarrow{} A(\Gamma)$ is a group embedding, then~$\Gamma$ contains an induced path~$P_4$.
\end{itemize}
\begin{remark}
Note that none of the groups~$A(P_4)$ and~$A(C_4)$ can be  a subgroup of the other, because none of graphs~$P_4$ and~$C_4$ is an induced subgraph of the other.
\end{remark}
Summarising the results above, we obtain the other direction of the goal as well, as desired.
\begin{comment}
~$\mathcal{A}(\mathcal{E}) = \mathcal{A} \setminus \mathcal{A}(P_4 \lor C_4)$, as desired.
\end{comment}

\subsubsection*{Separable RAAGs = Elementary RAAGs}

Word problem is one of the fundamental algorithmic questions in algebra; for a group~$G$, generated by a set~$A$, the \emph{word problem} asks whether~$w \in (A \cup A^{-1})^{\ast}$ is equal to~$1_G$ in~$G$.
This problem is solved for T-RAAGs in \citep{foniqi2022, foniqi2024twisted}.

A more general problem is the \emph{subgroup membership problem} (know also as the \emph{generalised word problem}) for a group~$G$, which asks for the existence of an algorithm that takes as input an element~$g \in G$ and a ﬁnitely generated subgroup~$H \leqslant G$ and outputs a YES -- NO answer whether~$g$ belongs to~$H$ or not. This is still open for RAAGs, hence for T-RAAGs as well. There is a useful connection between being subgroup separable and having decidable subgroup membership problem.

\begin{remark}[{\citealp{mal1958homomorphisms}}]
For ﬁnitely presented groups, being subgroup separable implies the solvability of the generalised word problem.
\end{remark}

In \citep{mikhailova1966occurrence} it is shown that~$A(C_4) = F_2 \times F_2$ contains a subgroup where membership is undecidable. 

\begin{cor}\label{cor: A(C_4) is not LERF}
The group~$A(C_4)$ is not subgroup separable.
\end{cor}
On the other hand~$A(P_4)$ has decidable subgroup membership problem, see \citep*[Corollary 1.3]{kapovich2005foldings}, as~$P_4$ is a chordal graph. However, \cite*{niblo2001subgroup} show that~$A(P_4)$ is not subgroup separable.
\begin{definition}\label{def: elementary RAAGs}
Simplicial graphs not containing~$P_4$ and~$C_4$ are called \emph{transitive forests}. Right-angled Artin groups based on transitive forests are called \emph{elementary RAAGs}.
\end{definition}
\begin{remark}\label{rem: elementary RAAGs}
The result \citep[Lemma 2]{lohrey2008submonoid} shows that the class of elementary RAAGs, denoted by~$\mathcal{A}(\mathcal{E})$, is described precisely with the following properties:
\begin{itemize}
\item[(i)] The trivial group~$\{1\}$ belongs to~$\mathcal{A}(\mathcal{E})$,
\item[(ii)] If~$E \in \mathcal{A}(\mathcal{E})$, then~$E \times \Z \in \mathcal{A}(\mathcal{E})$,
\item[(iii)] If~$E_1, E_2 \in \mathcal{A}(\mathcal{E})$ then~$E_1 \ast E_2 \in \mathcal{A}(\mathcal{E})$.
\end{itemize}
\end{remark}

%Let~$\Gamma$ be a simplicial graph. The group~$A(\Gamma)$ can have~$A(P_4)$, respectively~$A(C_4)$, as a subgroup if and only if one sees~$P_4$, respectively~$C_4$, as an induced subgraph of~$\Gamma$. \\
In particular, the results above demonstrate that the class of RAAGs that are LERF is a subclass of elementary RAAGs.
\begin{comment}
~$\mathcal{A}(SgSep) \subseteq \mathcal{A} \setminus \mathcal{A}(P_4 \lor C_4)$.
\end{comment}

On the other hand, the class of elementary RAAGs (see \autoref{rem: elementary RAAGs}) has nice closure properties for free products, and for direct products with~$\Z$.
%~$\mathcal{A} \setminus \mathcal{A}(P_4 \lor C_4) = \mathcal{A}(\mathcal{E})$, 
By closure properties of subgroup separability under direct products with~$\Z$ given by \autoref{cor: SgSep with direct products}, and free products given by \autoref{thm: SgSep with free products}, it follows that elementary RAAGs are subgroup separable, providing ultimately the proof of \autoref{thm: main theorem for RAAGs}.
%\autoref{eq: RAAGs - main_equation}, as desired.
%~$\mathcal{A}(\mathcal{E}) \subseteq \mathcal{A}(SgSep)$,

\section{Going from RAAGs to T-RAAGs}\label{Going from RAAGs to T-RAAGs}

When extending previous results from RAAGs to the class
%~$\mathcal{T}$ 
of T-RAAGs, which contains RAAGs, one searches for the 'poisonous pieces'~$A(P_4)$ and~$A(C_4)$ in the larger class. Since being subgroup separable
%~$SgSep$ 
is inherited by subgroups and it fails for~$A(P_4)$ and~$A(C_4)$, it fails for a T-RAAG~$G$
%~$G \in  \mathcal{T}$ 
as well whenever~$G$ contains~$A(P_4)$ or~$A(C_4)$.

We generalise the results from RAAGs to T-RAAGs by answering the following question.
\begin{question}\label{question: about containing P_4 and C_4}
Which T-RAAGs contain~$A(P_4)$ or~$A(C_4)$ as subgroups?
\end{question}
%Can one answer this quetion from \cite{pride1986tits}?
We will answer one part of this question
%\autoref{question: about containing P_4 and C_4} 
by applying the normal forms in T-RAAGs coming from \citep{foniqi2022, foniqi2024twisted}, which yield the following two results:

\begin{lemma}\label{lem: subgraph injectivity}
Let~$\Delta$ be an induced subgraph of the mixed graph~$\Gamma$. The morphism
$$i: T(\Delta) \longrightarrow T(\Gamma) \text{ induced by } v \mapsto v,$$
sends normal forms to normal forms, hence it is injective. In particular,~$T(\Delta) \leqslant T(\Gamma)$.
\end{lemma}

\begin{lemma}\label{lem: square injectivity}
The map~$q: A(\overline{\Gamma}) \longrightarrow T(\Gamma)$ given by~$v \mapsto v^2$ sends normal forms to normal forms, hence it is injective.
\end{lemma}
Note that for RAAGs, the last lemma comes from the classical result \citep[Corollary 2]{crisp2001solution}. In particular, one has:
\begin{cor}\label{cor: A(P_4) and A(C_4) injectivity}
If~$\overline{\Gamma}$ contains a path on~$4$ vertices, or a square on~$4$ vertices, then~$T(\Gamma)$ contains~$A(P_4)$ or~$A(C_4)$ respectively, in which case~$T(\Gamma)$ is not subgroup separable.
\end{cor}

\begin{remark}
The result of this corollary follows also from \cite[Theorem 2]{pride1986tits}, because T-RAAGs are "generalized" Artin groups with their convention, and they satisfy the~"$C(4)$ condition". Moreover, paths on~$4$ vertices, or a square on~$4$ vertices, are triangle-free, so they satisfy their~"$T(4)$ condition" as well. Therefore if~${\Gamma}$ contains a graph~$\Delta$ which is a path on~$4$ vertices, or a square on~$4$ vertices, then~$T(\Gamma)$ contains~$T(\Delta)$ by \autoref{lem: subgraph injectivity}; by \cite[Theorem 2]{pride1986tits}  the squares of vertices of~$\Delta$ generate a RAAG isomorphic to~$A(\overline{\Delta})$, in which case~$T(\Gamma)$ is not subgroup separable.
\end{remark}

\begin{cor}\label{cor: A not LERF implies T not LERF}
Let~$\Gamma$ be a mixed graph. If~$A(\overline{\Gamma})$ is not subgroup separable, then~$T(\Gamma)$ is not subgroup separable as well.
\end{cor}

To answer the other part of \autoref{question: about containing P_4 and C_4} we show that if~$\Gamma$ is a mixed graph such that its underlying simplicial graph~$\overline{\Gamma}$ does not contain neither~$P_4$ nor~$C_4$, then~$T(\Gamma)$ is subgroup separable; see \autoref{thm: main theorem}.
In particular,~$T(\Gamma)$ cannot contain~$A(P_4)$ or~$A(C_4)$ in this case.

\subsection{Subgroup Separability for T-RAAGs}\label{SgSep for T-RAAGs}
\begin{theorem}\label{thm: main theorem}
Let~$\Gamma$ be a mixed graph. Then~$T(\Gamma)$ is subgroup separable if and only if the underlying simplicial graph~$\overline{\Gamma}$ does not contain~$P_4$ and~$C_4$ as induced subgraphs.
\end{theorem}
\vspace{-0.15cm}
The theorem above is our main contribution on this article. Immediately, one has:
\begin{cor}\label{thm: main corollary}
$T(\Gamma)$ is subgroup separable if and only if~$A(\overline{\Gamma})$ is subgroup separable.
\end{cor}
\begin{proof}[Proof of \autoref{thm: main theorem}]
One direction is provided by \autoref{cor: A not LERF implies T not LERF}.

For the other direction, let~$\Gamma$ be a mixed graph, such that~$\overline{\Gamma}$ does not contain~$P_4$ and~$C_4$ as induced subgraphs, i.e.~$\overline{\Gamma}$ is a transitive forest. 
It suffices to show the result when~$\Gamma$ is connected: if~$\Gamma = \Gamma_1 \sqcup \Gamma_2$ is a disjoint union of two mixed graphs, then~$T(\Gamma) = T(\Gamma_1) \ast T(\Gamma_2)$, and subgroup separability is preserved under taking free products (see \autoref{thm: SgSep with free products}).

If~$|V\Gamma| = 1$, then~$A(\Gamma) \simeq \Z$ which is LERF.

If~$|V\Gamma| = 2$, then~$A(\Gamma) \simeq \Z^2$ which is subgroup separable from \autoref{cor: SgSep with direct products}, or~$T(\Gamma) \simeq K$ - Klein bottle group, which is a semi-direct product of two copies of~$Z$, hence is subgroup separable as well from \autoref{cor: SgSep with semi-direct products}. 

Now assume~$|V\Gamma| \geq 3$ and that all T-RAAGs based on graphs with less vertices than~$|V\Gamma|$ are subgroup separable.
Since~$\Gamma$ is connected and~$\overline{\Gamma}$ does not contain both~$P_4$ and~$C_4$, a Lemma in \citep{droms1987subgroups} shows that there is at least one vertex in~$\Gamma$, say~$x$, that is connected to every other vertex of~$\Gamma$. Hence, our graph~$\Gamma$ looks like this:
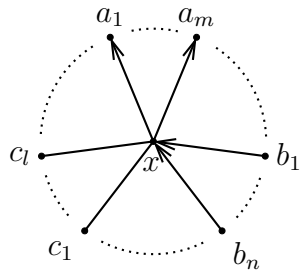
\begin{figure}[H]
\centering
\begin{tikzpicture}[>={Straight Barb[length=7pt,width=4pt]}, rotate = 22.5, scale = 0.75] 

\draw [dotted,domain=55:80, thick] plot ({2*cos(\x)}, {2*sin(\x)});
%% vertices & edge

\draw[fill=black] (0,0) circle (1.5pt) node[below] {};
\draw[fill=black] (-0.07,-0.07) circle (0pt) node[below] {$x$};

\draw[fill=black] ({2*cos(90)}, {2*sin(90)}) circle (1.5pt) node[above] {$a_1$};

\draw[fill=black] ({2*cos(45)}, {2*sin(45)}) circle (1.5pt) node[above] {$a_{m}$};

%% vertex labels
%%% edges
\draw[thick, ->] (0,0) -- ({2*cos(90)}, {2*sin(90)});
\draw[thick, ->] (0,0) -- ({2*cos(45)}, {2*sin(45)});

\draw [dotted,domain=295:320, thick] plot ({2*cos(\x)}, {2*sin(\x)});

\draw[fill=black] ({2*cos(330)}, {2*sin(330)}) circle (1.5pt) node[right] {$b_1$};

\draw[fill=black] ({2*cos(285)}, {2*sin(285)}) circle (1.5pt) node[below right] {$b_{n}$};

%% vertex labels
%%% edges
\draw[thick, <-] (0,0) -- ({2*cos(330)}, {2*sin(330)});
\draw[thick, <-] (0,0) -- ({2*cos(285)}, {2*sin(285)});

\draw [dotted,domain=175:200, thick] plot ({2*cos(\x)}, {2*sin(\x)});

\draw[fill=black] ({2*cos(165)}, {2*sin(165)}) circle (1.5pt) node[left] {$c_{l}$};

\draw[fill=black] ({2*cos(210)}, {2*sin(210)}) circle (1.5pt) node[below left] {$c_1$};

%% vertex labels
%%% edges
\draw[thick] (0,0) -- ({2*cos(165)}, {2*sin(165)});
\draw[thick] (0,0) -- ({2*cos(210)}, {2*sin(210)});

\draw [dotted,domain=-20:35, thick] plot ({2*cos(\x)}, {2*sin(\x)});
\draw [dotted,domain=100:155, thick] plot ({2*cos(\x)}, {2*sin(\x)});
\draw [dotted,domain=220:275, thick] plot ({2*cos(\x)}, {2*sin(\x)});
\end{tikzpicture}\caption{Star shaped graph}
\end{figure}\label{Star shaped graph}
where one can have other edges (directed or not) among the vertices on the dotted circle. The proof comes down to showing that~$T(\Gamma)$ is subgroup separable.
%$G$ has a finite index subgroup which is isomorphic to a RAAG based on a connected graph which does not contain~$P_4$ or~$C_4$ as induced subgraphs.
Consider the subgroup:
\[ 
G = \sgp{x^2, \lk(x)} = \sgp{x^2, a_1, \ldots, a_{m}, b_1, \ldots, b_{n}, c_1, \ldots, c_{l}}.
\]
{\bf Claim.} The group~$G$ is normal in~$T(\Gamma)$ of index~$2$. 

{\it Proof of Claim.} It is enough to show~$vGv^{-1} \subseteq G$ for any~$v \in V\Gamma$. This is clearly true, by definition of~$G$, for any~$v \in V\Gamma \setminus \{x\}$. On the other hand, for~$x$, it is enough to show that for any~$v \in V\Gamma \setminus \{x\}$ one has~$xvx^{-1} \in G$. We distinguish three cases:
\begin{itemize}[topsep=0pt, itemsep = 0pt]
\item[(i)] If~$[x, v] \in E\Gamma$, then~$xv = vx$ and so~$xvx^{-1} = v$, which lies in~$G$. 
\item[(ii)] If~$\langle x, v] \in E\Gamma$, then~$vxv = x$, hence~$xvx^{-1} = v^{-1}$ which lies in~$G$ because~$v \in G$. 
\item[(iii)] If~$[x, v \rangle \in E\Gamma$, then~$xvx = v$, giving~$xvx^{-1} = vx^{-2}$ which lies in~$G$ because~$v, x^{2} \in G$.  
\end{itemize}
This completes the proof that~$G$ is normal in~$T(\Gamma)$. One can see that~$G$ has index~$2$ in~$T(\Gamma)$, by considering the quotient~$T(\Gamma) / G = \gp{x}{x^2}$ which is a cyclic group of order~$2$.

Now we can apply the Reidemeister-Schreier procedure to obtain a presentation for~$G$. \\
{\bf Claim.} The group~$G$ has a T-RAAG presentation based on the graph~$\Delta$ given below:
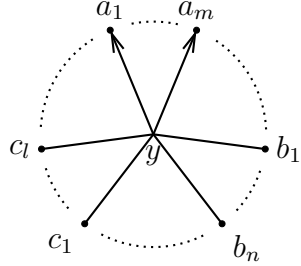
\begin{figure}[H]
\centering
\begin{tikzpicture}[>={Straight Barb[length=7pt,width=4pt]}, rotate = 22.5, scale = 0.75] 

\draw [dotted,domain=55:80, thick] plot ({2*cos(\x)}, {2*sin(\x)});
%% vertices & edge

%\draw[fill=black] (0,0) circle (1.5pt) node[below] {};
\draw[fill=black] (0,0) circle (0pt) node[below] {$y$};

\draw[fill=black] ({2*cos(90)}, {2*sin(90)}) circle (1.5pt) node[above] {$a_1$};

\draw[fill=black] ({2*cos(45)}, {2*sin(45)}) circle (1.5pt) node[above] {$a_{m}$};

%% vertex labels
%%% edges
\draw[thick, ->] (0,0) -- ({2*cos(90)}, {2*sin(90)});
\draw[thick, ->] (0,0) -- ({2*cos(45)}, {2*sin(45)});

\draw [dotted,domain=295:320, thick] plot ({2*cos(\x)}, {2*sin(\x)});

\draw[fill=black] ({2*cos(330)}, {2*sin(330)}) circle (1.5pt) node[right] {$b_1$};

\draw[fill=black] ({2*cos(285)}, {2*sin(285)}) circle (1.5pt) node[below right] {$b_{n}$};

%% vertex labels
%%% edges
\draw[thick] (0,0) -- ({2*cos(330)}, {2*sin(330)});
\draw[thick] (0,0) -- ({2*cos(285)}, {2*sin(285)});

\draw [dotted,domain=175:200, thick] plot ({2*cos(\x)}, {2*sin(\x)});

\draw[fill=black] ({2*cos(165)}, {2*sin(165)}) circle (1.5pt) node[left] {$c_{l}$};

\draw[fill=black] ({2*cos(210)}, {2*sin(210)}) circle (1.5pt) node[below left] {$c_1$};

%% vertex labels
%%% edges
\draw[thick] (0,0) -- ({2*cos(165)}, {2*sin(165)});
\draw[thick] (0,0) -- ({2*cos(210)}, {2*sin(210)});

\draw [dotted,domain=-20:35, thick] plot ({2*cos(\x)}, {2*sin(\x)});
\draw [dotted,domain=100:155, thick] plot ({2*cos(\x)}, {2*sin(\x)});
\draw [dotted,domain=220:275, thick] plot ({2*cos(\x)}, {2*sin(\x)});
\end{tikzpicture}
\end{figure}
where~$y$ corresponds to~$x^2$, and the dotted region is the same as the one in~$\Gamma$.

{\it Proof of Claim.} With reference to \Cref{R-S}, the set~$T = \{1, x\}$ is a Schreier transversal for~$H$ in~$G$. Note that we obtain~$S = \{x^2, a_1, \ldots, a_{m}, b_1, \ldots, b_{n}, c_1, \ldots, c_{l}\}$ as a generating set, because the other generators of the form~$xv\overline{xv}^{-1}$ for~$v \in V\Gamma \setminus \{x\}$ are products of elements of~$S$ and their inverses - with the same proof as for the previous claim. \\ 
Furthermore, the relations we obtain now between elements in~$S$ are as depicted in the graph~$\Delta$; the only change is that~$y = x^2$ commutes with all the~$b_i$'s. Note that the edges belonging on the dotted region, do not change at all. This finishes the proof of the claim.

Using \autoref{thm: SgSep with subgroups and finite extensions}, it is enough to show that~$G$ is subgroup separable, as a finite index subgroup of~$T(\Gamma)$, which would imply that~$T(\Gamma)$ is subgroup separable as well.

{\bf Claim.} The group~$G = T(\Delta)$ is a semi-direct product of~$H = T(\Delta \setminus \{y\})$ acting on the normal subgroup~$N = \sgp{y} \simeq \Z$.
%$G = N \rtimes H$.

{\it Proof of Claim.} First, the subgroup~$N = \sgp{y}$ is normal, because for any~$v \in V\Delta \setminus \{y\}$ we either have~$[y, v] \in E\Delta$ or~$[y, v\rangle \in E\Delta$; the first edge gives~$vyv^{-1} = y \in N$, while the second (directed) edge gives~$vyv^{-1} = y^{-1} \in N$.

One also has~$N \cap H = \{1\}$ as intersection of standard parabolic subgroups, coming from normal forms in \citep{foniqi2024twisted}. 

Moreover~$G = N \cdot H$, because~$y$ shuffles with any generator of~$H$; indeed, for any~$v \in V\Delta \setminus \{y\}$ we either have~$vy = yv$ or~$vy = y^{-1}v$. Ultimately,~$G = N \rtimes H \simeq \Z \rtimes H$.

By induction, the group~$T(\Delta \setminus \{y\})$ is subgroup separable because~$\Delta \setminus \{y\}$ 
%is a disjoint union of connected mixed graphs with 
has less vertices than~$|V\Gamma|$. Ultimately,~$G$ is subgroup separable as a semi-direct product of~$\Z$ and a subgroup separable group (this follows from \autoref{cor: SgSep with semi-direct products}).
\end{proof}

\begin{comment}

$ ==============================================~$ 

Note that if two vertices~$a, b$ commute, their squares~$a^2, b^2$ commute as well, so it is reasonable to have~$a^2 b^2 = b^2 a^2$ as a relation. On the other hand, if~$a, b$ are connected by an edge with label~$p > 2$, the result shows that the corresponding squares~$a^2, b^2$ are free of relations. 

By abuse of notation, we use a graphical presentation for this subgroup as below:
\begin{figure}[H]
\centering
\begin{tikzpicture}[scale = 0.8, >={Straight Barb[length=7pt,width=6pt]},thick]

\draw[fill=black] (2,1) circle (1pt) node[left] {$a^2$};
\draw[fill=black] (4,1) circle (1pt) node[right] {$b^2$};
\draw[fill=black] (4,-1) circle (1pt) node[right] {$c^2$};
\draw[fill=black] (2,-1) circle (1pt) node[left] {$d^2$};

\draw[] (1, 0) node[left] {$\sgp{a^2, \, b^2, \, c^2, \, d^2} \simeq~$};
%% vertices & nodes

%%% edges
\draw[thick] (2, 1) -- (4, 1);
\draw[thick] (4,1) -- (4,-1);
\draw[thick] (4,-1) -- (2,-1);
\draw[thick] (2,-1) -- (2,1);
\end{tikzpicture}
\end{figure}
\end{comment}

\section{Membership Problem in T-RAAGs}\label{Membership Problem in T-RAAGs}

A \emph{decision problem} is a YES -- NO question on a countable infinite set of inputs. 
A decision problem is called \emph{decidable} if there is an algorithm which takes an input, terminates after a finite amount of time, and correctly answers the question by YES or NO. If such an algorithm does not exist, we refer to the decision problem as \emph{undecidable}.

The word problem and the subgroup membership problem, introduced earlier, are both instances of decision problems. Natural generalizations of these classical problems include the submonoid membership problem and the rational subset membership problem, which we will define and discuss in more detail in~\Cref{Submonoids and Rational Subsets}.

\subsection{Subgroup Membership Problem in T-RAAGs}\label{Subgroup Membership Problem in T-RAAGs}

The subgroup membership problem remains open for RAAGs, and hence it is open for T-RAAGs as well. However, subgroup separable T-RAAGs have a decidable subgroup membership problem~\citep{mal1958homomorphisms}.

On the other hand, subgroup non-separability does not imply undecidability of the subgroup membership problem: for example, \(A(P_4)\) is not subgroup separable, yet has a decidable subgroup membership problem.

\begin{prop}[{\citealp*[Corollary 1.3]{kapovich2005foldings}}] Let $\Gamma$ be a simplicial chordal graph. Then,~$A(\Gamma)$ has decidable subgroup membership problem.
\end{prop}

In contrast, the RAAG \(A(C_4) = F_2 \times F_2\) contains a subgroup with undecidable membership~\citep{mikhailova1966occurrence}; the graph \(C_4\) is not chordal.

RAAGs defined by chordal graphs are precisely the coherent ones~\citep{droms1987subgroups}.

\begin{cor}
Coherent RAAGs have decidable subgroup membership problem.
\end{cor}

For a finite mixed graph~$\Gamma$ we will say that~$\Gamma$ is \emph{chordal} if~$\overline{\Gamma}$ is chordal.

One also has that T-RAAGs based on chordal graphs are exactly the coherent ones (see [Blumer, Foniqi, Quadrelli]).

The main goal of this secion, which uses \citep*[Theorem 1.1]{kapovich2005foldings}, is to show

\begin{prop}\label{prop: subgroup membership} Let $\Gamma$ be a mixed chordal graph. Then,~$T(\Gamma)$ has decidable subgroup membership problem.
\end{prop}

\begin{proof}
Let~$\Gamma$ be chordal. In the case when~$\Gamma$ is complete, the group~$T(\Gamma)$ is virtually free abelian; indeed, the subgroup generated by squares of the vertices of $\Gamma$ is free abelian, and of finite index. Threfore, $T(\Gamma)$ has decidable subgroup membership problem. 

If~$\Gamma$ is not complete, then there are two proper subgraphs~$\Gamma_1$ and~$\Gamma_2$ with~$\Gamma = \Gamma_1 \cup \Gamma_2$ and~$\Gamma_1 \cap \Gamma_2 = C$, with~$C$ complete. Here one has a splitting:
\[
T(\Gamma) = T(\Gamma_1) *_ {T(C)} T(\Gamma_2).
\]
As~$C$ is complete,~$T(C)$ is virtually free abelian. 

Now, by the induction hypothesis, both $T(\Gamma_1)$, and $T(\Gamma_2)$ are iterated amalgamated products or graphs of groups - with virtually abelian vertex groups (and edge groups). 
Now the result follows from induction and \citep*[Theorem 1.1]{kapovich2005foldings}.
\end{proof}

\subsection{Submonoids and Rational Subsets}\label{Submonoids and Rational Subsets}

Let $M$ be a monoid generated by a finite set $A$, and let $\phi: A^* \to M$ be the canonical homomorphism.

\medskip
\noindent
The \emph{submonoid membership problem} asks: given a finite set $S \subseteq A^*$ and a word $w \in A^*$, does $\phi(w) \in \phi(S^*)$? This problem is well-defined for groups as well, using monoid generators $A \cup A^{-1}$.

\medskip
\noindent
A subset $L \subseteq M$ is \emph{rational} if $L = \phi(R)$ for some regular language $R \subseteq A^*$. Rational subsets generalize finitely generated submonoids.

\medskip
\noindent
The \emph{rational subset membership problem} asks: given a rational subset $R \subseteq M$ and a word $w \in A^*$, is $\phi(w) \in R$? This applies to groups viewed as monoids.

\begin{remark}
Decidability of rational subset membership implies decidability of submonoid membership.
\end{remark}

\medskip
\noindent
These problems are preserved under passing to submonoids. Notably, a group may have decidable submonoid membership but undecidable rational subset membership~\citep{bodart2024membership}.

\begin{theorem}[{\citealp{lohrey2008submonoid}}]
A RAAG $A(\Gamma)$ has decidable submonoid and rational subset membership problems if and only if $\Gamma$ contains neither $P_4$ nor $C_4$ as induced subgraphs.
\end{theorem}

It is clear that if $\Gamma$ is a mixed graph of non-elementary type, then $T(\Gamma)$ has undecidable submonoid membership problem, as it contains $A(P_4)$ or $A(C_4)$.

On the other hand, the situation is not clear for elementary mixed graphs. But, there is an interesting sub-family of elementary T-RAAGs for which the rational subset membership problem is decidable; the construction of this family will be given below and it is motivated from the proof of \autoref{thm: main theorem}.

\begin{definition}\label{def: sink-cone}
Let $\Gamma = (V, E)$ be a mixed graph. The \emph{$\Gamma$-cone} $\nabla(\Gamma)$, with tip $w$, is a mixed graph with vertex set $V \sqcup \{w\}$, such that $\Gamma$ is an induced subgraph, and every vertex~$v$ of $\Gamma$ is connected to $w$ with edges of the type $[v, w]$ or $[v, w \rangle$.
\end{definition}

\begin{definition}\label{def: class R}
Let $\mathcal{R}$ be the class of mixed graphs of elementary type, constructed by starting with a single vertex, and iterating through disjoint unions and $\Gamma$-cones.
\end{definition}

\begin{theorem}\label{thm: class R}
Let $\Gamma$ be a mixed graph belonging in $\mathcal{R}$. $T(\Gamma)$ has decidable rational subset membership problem. Consequently, $T(\Gamma)$ has decidable submonoid membership problem.
\end{theorem}

\begin{proof}
We proceed by induction on $n = |V\Gamma|$.

If $n = 1$, then $T(\Gamma) \simeq \Z$, and the result is immediate.

For $n = 2$, there are three possibilities:
\begin{itemize}
  \item[(i)] $T(\Gamma) \simeq \Z \ast \Z$, which has a decidable rational subset membership problem, as a free product of groups that do.
  \item[(ii)] $T(\Gamma) \simeq \Z \times \Z$, which has a decidable rational subset membership problem, being a direct product of a group that does with $\Z$.
  \item[(iii)] $T(\Gamma) \simeq \Z \rtimes \Z$, which also has a decidable rational subset membership problem, as a finite extension of $\Z^2$.
\end{itemize}
Now assume that $n > 2$ and that the result holds for graphs in $\mathcal{R}$ with less vertices. 

If $\Gamma$ is disconnected, then $\Gamma = \Gamma_1 \sqcup \Gamma_2$, and so $T(\Gamma) = T(\Gamma_1) \ast T(\Gamma_2)$ and the result holds by induction and the closere of decidability of rational subset membership prblem under free products.

So, assume that $\Gamma$ is connected. By the construction of $\mathcal{R}$, we know that there is a vertex $w$ connected with every other vertex $v \in V\Gamma \setminus \{w\}$, by edges of the type $[v, w]$ or $[v, w \rangle$. Since graphs in $\mathcal{R}$ are mixed elementary as well, we can apply the proof of \autoref{thm: main theorem}. In particular, our graph looks like the one in \autoref{Star shaped graph}, without outgoing edges, hence it looks graphically as below:
\begin{figure}[H]
\centering
\begin{tikzpicture}[>={Straight Barb[length=7pt,width=4pt]}, rotate = 53.5, scale = 0.75] 

\draw [dotted,domain=55:80, thick] plot ({2*cos(\x)}, {2*sin(\x)});
%% vertices & edge

\draw[fill=black] (0,0) circle (1.5pt) node[below] {};
\draw[fill=black] (-0.07,-0.07) circle (0pt) node[left] {$w$};

\draw[fill=black] ({2*cos(90)}, {2*sin(90)}) circle (1.5pt) node[above] {$b_1$};

\draw[fill=black] ({2*cos(45)}, {2*sin(45)}) circle (1.5pt) node[above] {$b_{2}$};

%% vertex labels
%%% edges
\draw[thick, <-] (0,0) -- ({2*cos(90)}, {2*sin(90)});
\draw[thick, <-] (0,0) -- ({2*cos(45)}, {2*sin(45)});

\draw [dotted,domain=295:320, thick] plot ({2*cos(\x)}, {2*sin(\x)});

\draw[fill=black] ({2*cos(330)}, {2*sin(330)}) circle (1.5pt) node[right] {$b_n$};

\draw[fill=black] ({2*cos(285)}, {2*sin(285)}) circle (1.5pt) node[below right] {$c_{l}$};

%% vertex labels
%%% edges
\draw[thick, <-] (0,0) -- ({2*cos(330)}, {2*sin(330)});
\draw[thick, -] (0,0) -- ({2*cos(285)}, {2*sin(285)});

\draw [dotted,domain=175:200, thick] plot ({2*cos(\x)}, {2*sin(\x)});

\draw[fill=black] ({2*cos(165)}, {2*sin(165)}) circle (1.5pt) node[left] {$c_{1}$};

\draw[fill=black] ({2*cos(210)}, {2*sin(210)}) circle (1.5pt) node[below left] {$c_l$};

%% vertex labels
%%% edges
\draw[thick] (0,0) -- ({2*cos(165)}, {2*sin(165)});
\draw[thick] (0,0) -- ({2*cos(210)}, {2*sin(210)});

\draw [dotted,domain=-20:35, thick] plot ({2*cos(\x)}, {2*sin(\x)});
\draw [dotted,domain=100:155, thick] plot ({2*cos(\x)}, {2*sin(\x)});
\draw [dotted,domain=220:275, thick] plot ({2*cos(\x)}, {2*sin(\x)});
\end{tikzpicture}
\end{figure}
Now, as in the proof of \autoref{thm: main theorem}, the subgroup:
\[ 
G = \sgp{w^2, \lk(w)} = \sgp{w^2, b_1, \ldots, b_{n}, c_1, \ldots, c_{l}}.
\]
is a T-RAAG over the mixed graph
\begin{figure}[H]
\centering
\begin{tikzpicture}[>={Straight Barb[length=7pt,width=4pt]}, rotate = 53.5, scale = 0.75] 

\draw [dotted,domain=55:80, thick] plot ({2*cos(\x)}, {2*sin(\x)});
%% vertices & edge

\draw[fill=black] (0,0) circle (1.5pt) node[below] {};
\draw[fill=black] (-0.07,-0.07) circle (0pt) node[left] {$w^2$};

\draw[fill=black] ({2*cos(90)}, {2*sin(90)}) circle (1.5pt) node[above] {$b_1$};

\draw[fill=black] ({2*cos(45)}, {2*sin(45)}) circle (1.5pt) node[above] {$b_{2}$};

%% vertex labels
%%% edges
\draw[thick] (0,0) -- ({2*cos(90)}, {2*sin(90)});
\draw[thick] (0,0) -- ({2*cos(45)}, {2*sin(45)});

\draw [dotted,domain=295:320, thick] plot ({2*cos(\x)}, {2*sin(\x)});

\draw[fill=black] ({2*cos(330)}, {2*sin(330)}) circle (1.5pt) node[right] {$b_n$};

\draw[fill=black] ({2*cos(285)}, {2*sin(285)}) circle (1.5pt) node[below right] {$c_{l}$};

%% vertex labels
%%% edges
\draw[thick] (0,0) -- ({2*cos(330)}, {2*sin(330)});
\draw[thick, -] (0,0) -- ({2*cos(285)}, {2*sin(285)});

\draw [dotted,domain=175:200, thick] plot ({2*cos(\x)}, {2*sin(\x)});

\draw[fill=black] ({2*cos(165)}, {2*sin(165)}) circle (1.5pt) node[left] {$c_{1}$};

\draw[fill=black] ({2*cos(210)}, {2*sin(210)}) circle (1.5pt) node[below left] {$c_l$};

%% vertex labels
%%% edges
\draw[thick] (0,0) -- ({2*cos(165)}, {2*sin(165)});
\draw[thick] (0,0) -- ({2*cos(210)}, {2*sin(210)});

\draw [dotted,domain=-20:35, thick] plot ({2*cos(\x)}, {2*sin(\x)});
\draw [dotted,domain=100:155, thick] plot ({2*cos(\x)}, {2*sin(\x)});
\draw [dotted,domain=220:275, thick] plot ({2*cos(\x)}, {2*sin(\x)});
\end{tikzpicture}
\end{figure}
In particular, $G = T(\Gamma \setminus \{w\}) \times \Z$, which has decidable rational subset membership problem, by iunduction and the closere under direct product with $\Z$.

On the other hand, $G$ is a subgroup of index $2$ in $T(\Gamma)$, and so $T(\Gamma)$ has decidable rational subset membership problem, by closere under finite extensions.
\end{proof}

\noindent\textit{\\ Islam Foniqi,\\
The University of East Anglia\\ 
Norwich (United Kingdom)\\}
{email: i.foniqi@uea.ac.uk}

\setlength{\bibsep}{7pt}
\bibliography{main}

\begin{thebibliography}{20}
\providecommand{\natexlab}[1]{#1}
\providecommand{\url}[1]{\texttt{#1}}
\expandafter\ifx\csname urlstyle\endcsname\relax
  \providecommand{\doi}[1]{doi: #1}\else
  \providecommand{\doi}{doi: \begingroup \urlstyle{rm}\Url}\fi

\bibitem[Allenby and Gregorac(2006)]{allenby2006locally}
R.~Allenby and R.~J. Gregorac.
\newblock On locally extended residually finite groups.
\newblock In \emph{Conference on Group Theory: University of Wisconsin-Parkside
  1972}, pages 9--17. Springer, 2006.

\bibitem[Bodart(2024)]{bodart2024membership}
C.~Bodart.
\newblock Membership problems in nilpotent groups.
\newblock \emph{arXiv preprint arXiv:2401.15504}, 2024.

\bibitem[Burns(1971)]{burns1971finitely}
R.~G. Burns.
\newblock On finitely generated subgroups of free products.
\newblock \emph{Journal of the Australian Mathematical Society}, 12\penalty0
  (3):\penalty0 358--364, 1971.

\bibitem[Charney(2007)]{charney2007introduction}
R.~Charney.
\newblock An introduction to right-angled {Artin} groups.
\newblock \emph{Geometriae Dedicata}, 125\penalty0 (1):\penalty0 141–158,
  2007.

\bibitem[Clancy and Ellis(2010)]{clancy2010homology}
M.~Clancy and G.~Ellis.
\newblock Homology of some {Artin} and twisted {Artin} groups.
\newblock \emph{Journal of K-Theory}, 6\penalty0 (1):\penalty0 171–196, 2010.

\bibitem[Crisp and Paris(2001)]{crisp2001solution}
J.~Crisp and L.~Paris.
\newblock The solution to a conjecture of {Tits} on the subgroup generated by
  the squares of the generators of an {Artin} group.
\newblock \emph{Inventiones mathematicae}, 145:\penalty0 19--36, 2001.

\bibitem[Droms(1987)]{droms1987subgroups}
C.~Droms.
\newblock Subgroups of graph groups.
\newblock \emph{J. ALGEBRA.}, 110\penalty0 (2):\penalty0 519--522, 1987.

\bibitem[Foniqi(2022)]{foniqi2022}
I.~Foniqi.
\newblock \emph{Results on {Artin} and twisted {Artin} groups}.
\newblock PhD thesis, University of Milano - Bicocca, 2022.
\newblock URL \url{https://boa.unimib.it/handle/10281/374264}.

\bibitem[Foniqi(2024)]{foniqi2024twisted}
I.~Foniqi.
\newblock Twisted right-angled artin groups.
\newblock \emph{arXiv preprint arXiv:2407.06933}, 2024.

\bibitem[Kambites(2009)]{kambites2009commuting}
M.~Kambites.
\newblock On commuting elements and embeddings of graph groups and monoids.
\newblock \emph{Proceedings of the Edinburgh Mathematical Society}, 52\penalty0
  (1):\penalty0 155--170, 2009.

\bibitem[Kapovich et~al.(2005)Kapovich, Weidmann, and
  Myasnikov]{kapovich2005foldings}
I.~Kapovich, R.~Weidmann, and A.~Myasnikov.
\newblock Foldings, graphs of groups and the membership problem.
\newblock \emph{International Journal of Algebra and Computation}, 15\penalty0
  (01):\penalty0 95--128, 2005.

\bibitem[Kim and Koberda(2013)]{kim2013embedability}
S.-h. Kim and T.~Koberda.
\newblock Embedability between right-angled {Artin} groups.
\newblock \emph{Geometry \& Topology}, 17\penalty0 (1):\penalty0 493--530,
  2013.

\bibitem[Lohrey and Steinberg(2008)]{lohrey2008submonoid}
M.~Lohrey and B.~Steinberg.
\newblock The submonoid and rational subset membership problems for graph
  groups.
\newblock \emph{Journal of Algebra}, 320\penalty0 (2):\penalty0 728--755, 2008.

\bibitem[Lyndon and Schupp(1977)]{lyndon1977combinatorial}
R.~C. Lyndon and P.~E. Schupp.
\newblock \emph{Combinatorial group theory}, volume 188.
\newblock Springer, 1977.

\bibitem[Mal’cev(1958)]{mal1958homomorphisms}
A.~I. Mal’cev.
\newblock On homomorphisms onto finite groups.
\newblock \emph{Fluchen. Zap. Ivanovskogo Gos. Ped. Inst}, 18:\penalty0 49--60,
  1958.

\bibitem[Metaftsis and Raptis(2008)]{metaftsis2008profinite}
V.~Metaftsis and E.~Raptis.
\newblock On the profinite topology of right-angled {Artin} groups.
\newblock \emph{Journal of Algebra}, 320\penalty0 (3):\penalty0 1174--1181,
  2008.

\bibitem[Mikhailova(1966)]{mikhailova1966occurrence}
K.~Mikhailova.
\newblock The occurrence problem for direct products of groups.
\newblock \emph{Matematicheskii Sbornik}, 112\penalty0 (2):\penalty0 241--251,
  1966.

\bibitem[Niblo and Wise(2001)]{niblo2001subgroup}
G.~Niblo and D.~Wise.
\newblock Subgroup separability, knot groups and graph manifolds.
\newblock \emph{Proceedings of the American Mathematical Society}, 129\penalty0
  (3):\penalty0 685--693, 2001.

\bibitem[Pride(1986)]{pride1986tits}
S.~J. Pride.
\newblock On tits' conjecture and other questions concerning artin and
  generalized artin groups.
\newblock \emph{Inventiones mathematicae}, 86\penalty0 (2):\penalty0 347--356,
  1986.

\bibitem[Scott(1978)]{scott1978subgroups}
P.~Scott.
\newblock Subgroups of surface groups are almost geometric.
\newblock \emph{Journal of the London Mathematical Society}, 2\penalty0
  (3):\penalty0 555--565, 1978.

\end{thebibliography}

%\bibliographystyle{plain}
%\end{thebibliography}

\end{document}